\documentclass[english]{article}
\usepackage{amsfonts}
\usepackage{amsmath}
\usepackage{amssymb}
\usepackage{amscd}
\usepackage{amsthm}
\allowdisplaybreaks
\newtheorem{theorem}{Theorem}[section]
\newtheorem{proposition}{Proposition}[section]

\newtheorem{lemma} {Lemma}[section]

\newtheorem{remark}{Remark}[section]

\newtheorem{definition} {{Definition}}[section]

\def\R{{\mathbb R}}
\newcommand{\be} {\begin{equation}}
\newcommand{\ee} {\end{equation}}
\newcommand{\bea} {\begin{eqnarray}}
\newcommand{\eea} {\end{eqnarray}}
\newcommand{\Bea} {\begin{eqnarray*}}
\newcommand{\Eea} {\end{eqnarray*}}

\newcommand{\essinf}{\mathop{\rm {ess\,inf}}\limits}

\begin{document}

\title{\bf \sc On weak solutions for fourth-order problems involving the Leray-Lions type operators}
\date{}
\maketitle
\begingroup\small
\begin{center}
{\large \bf \sc
 K. Kefi, D.D. Repov\v{s} and K. Saoudi
}
\vskip 1.0cm
\end{center}
\endgroup

\begin{abstract}
%\noindent
We investigate existence and  multiplicity of weak solutions for fourth-order problems involving the Leray-Lions type operators in variable exponent spaces and improve a result of Bonanno and Chinn\`{i} (2011). We use variational methods and apply a multiplicity theorem of Bonanno and Marano (2010).
\end{abstract}



\section{Introduction }
 The objective of this work is to study   the existence of  solutions for the following problems involving the Leray-Lions type operators in variable exponent spaces
 \begin{equation}\label{P}
 \left\{
 \begin{array}{cc}
 \Delta\left(a(x, \Delta u)\right) = \lambda f(x,u)\;\;\mbox{in}\;\;  \Omega,\\
  u=\Delta u=0\,\,\, \mbox{ on}\,\,\, \partial{\Omega},
 \end{array}
 \right.
 \end{equation}
 where $\Omega$ is a bounded domain in $\R^{N\geq2}$ with  smooth boundary $\partial\Omega$, $\lambda>0$ is a parameter, $f$ is a Carath\'eodory  function,   $ p\in C(\overline\Omega)$ satisfies $\inf_{x\in\Omega}{p(x)}>\frac{N}{2}$, for all $x\in\Omega$, and the potential $a$ satisfies a set of conditions (see Section \ref{Assump}).  The  operators include the $p(x)$-biharmonic operator and other important cases.
 We point out that the extension from the $p$-biharmonic problem to the $p(x)$-biharmonic problem is nontrivial since
 the $p(x)$-biharmonic problem possesses a more complicated structure, for example it is nonhomogeneous and  it usually does not have the so-called first eigenvalue.
 Here, $ \Delta(a(x,\Delta u))$   is  the Leray-Lions operator, where  $a$ is a Carath\'{e}odory function satisfying  some  suitable supplementary conditions.
 
 Investigations of this type of operators has been going on in various fields, e.g. 
 in electrorheological fluids 
 (see  Ru\v{z}i\v{c}ka~\cite{ruzi}), 
 elasticity theory
 (see   Zhikov~\cite{zh2}),
 stationary thermorheological viscous flows of non-Newtonian fluids
 (see  Antontsev and Rodrigues~\cite{AR}), 
 image processing  
 (see  Chen, Levine and Rao~\cite{R00}),
  and 
   mathematical description of the processes filtration of barotropic gas through a porous medium  
   (see  Antontsev and Shmarev~\cite{AN}).
   For more details about this kind of operators the reader is referred to  
  Leray and Lions~\cite{Leray}
 (see  also
 Papageorgiou, R\u{a}dulescu and Repov\v{s}~\cite{PRR}
  and
 the references therein).
   
 We briefly recall the literature concerning
 related  problems involving the Leray-Lions type operators. The existence of three solutions for a problem involving
 the
  $p(x)
 $-Laplacian in a variable space was established by  
 Bonanno
 and Chinn\`{i}~\cite{BoCh}. 
 In particular, in the absence of
 small perturbations of the nonlinear term, they proved that one of the solutions is the trivial solution.  
 Bonanno and Chinn\`{i}~\cite{BoCh1}  also
 proved
 the existence of three solutions for a
 problem without small perturbations of the nonlinear term, whenever
 $p(x)>N.$ Motivated by these  results, we shall establish in this paper
 the existence of three nontrivial solutions for more general problems  involving the Leray-Lions type operators in variable exponent spaces (see  Theorem~\ref{principal} in Section~\ref{s3}).
 
 This paper is organized as follows: in Section~\ref{s2} we present preliminary definitions and results,
 Section~\ref{s3} is devoted to the statement and the proof proof of the main result, and Section~\ref{s4} presents an application.
 
 \section{Preliminaries}\label{s2}
 In this section we first review key definitions and basic properties of
 variable exponent Sobolev spaces and then introduce all conditions for parameters and nonlinearities which we need for the statement and proof of our main result (Theorem~\ref{principal}).
 \subsection{Variable exponent Sobolev spaces}\label{sec2} For a comprehensive treatment  we
 refer the reader to 
    R\u{a}dulescu and Repov\v{s}~\cite{RR}.
 Let $$C_{+}(\overline{\Omega}):=\{h\mid h\in C(\overline{\Omega}), h(x)>1,\ \mbox{for all}\;\; x\in \overline{\Omega}\},$$
 and let  $p\in C_{+}(\overline{\Omega}), p\in C_{+}(\overline{\Omega})$ be such that
 \begin{equation}\label{e2.2}
 1<p^{-}:=\displaystyle\min_{x\in\overline{\Omega}}p(x)\leq
 p^{+}:=\displaystyle\max_{x\in\overline{\Omega}}p(x)<+\infty.\end{equation}
 We define the Lebesgue space with variable exponent as  follows
 $$L^{p(x)}(\Omega):=\{u\mid u:\Omega\rightarrow \mathbb{R}\;\;\mbox{is measurable,} \int_{\Omega}|u(x)|^{p(x)}dx<\infty\},$$
 and we equip it with the  Luxemburg norm  
 $$|u|_{p(x)}:=\inf\left\{\mu>0\mid \int_{\Omega}\left|\frac{u(x)}{\mu}\right|^{p(x)}dx\leq 1\right\}.$$
 
 Variable exponent Lebesgue spaces are like classical Lebesgue spaces in many respects: they
 are Banach spaces
 and
  they are reflexive if and only if $1 < p^{-}\leq
 q^{+} < \infty$.  Moreover, the inclusion between Lebesgue spaces is generalized naturally, that is, if $q_{1}, q_{2}$ are such that
 $p_{1}(x) \leq p_{2}(x)$ a.e. $x \in \Omega$, then there exists a continuous embedding
 $L^{{p_{2}}(x)}(\Omega)\hookrightarrow L^{{p_{1}}(x)}(\Omega)$.
 
 For every $u \in L^{p(x)}(\Omega), v \in L^{p'(x)}(\Omega)$, the H\"older inequality
 \begin{equation}\label{ing}
 \left|\int_{\Omega}u v dx\right|\leq \left(\frac{1}{p^{-}}+\frac{1}{(p')^{-}}\right)|u|_{p(x)}|v|_{p'(x)}
 \end{equation}
 holds, where $p'(x)$ satisfies $\frac{1}{p(x)} + \frac{1}{p'(x)}=1$.
 The modular on the space $L^{p(x)}(\Omega)$ is the map $\rho_{p(x)}:L^{p(x)}(\Omega)\rightarrow \mathbb{R},$ defined by $$\rho_{p(x)}(u):=\int_{\Omega}|u|^{p(x)}dx.$$
 
 For any $k\in\mathbb{N},$ we define the Sobolev space with variable exponents as follows
 \begin{eqnarray*} W^{k,p(x)}(\Omega):=\{u\in L^{p(x)}(\Omega)\mid D^{\alpha}u\in L^{p(x)}(\Omega), |\alpha|\leq k\},
 \end{eqnarray*}
 where $\alpha := \left(\alpha_1, \alpha_2, . . . , \alpha_n\right)$ is a multi-index, $|\alpha|=\sum_{i=1}^{n}\alpha_i,$
 and $$D^{\alpha}u:=\frac{\partial^{|\alpha|}u}{\partial^{\alpha_1}x_1....\partial^{\alpha_N}x_N}.$$
 Then $W^{k,p(x)}(\Omega)$ is a separable and reflexive Banach space equipped with the norm
 $$\|u\|_{k,p(x)}:=\sum_{|\alpha|\leq k}|D^{\alpha}u|_{p(x)}.$$
 The space $W_{0}^{k,p(x)}(\Omega)$ is the closure of $C_{0}^{\infty}(\Omega)$ in $W^{k,p(x)}(\Omega)$.
 It is well-known that both $W^{2,p(x)}(\Omega)$ and $W_{0}^{1,p(x)}(\Omega)$ are separable and reflexive Banach spaces 
 (for more details see R\u{a}dulescu and Repov\v{s}~\cite{RR}).
 
 It follows that $X=W^{2,p(x)}(\Omega)\cap W_{0}^{1,p(x)}(\Omega)$
 is also a separable and reflexive Banach space, equipped with the norm
 $$\|u\|_{X}:=\|u\|_{W^{2,p(x)}(\Omega)}+\|u\|_{W_{0}^{1,p(x)}(\Omega)}.$$
 Let $$\|u\|:=\inf\left\{ \mu>0\mid \int_{\Omega}\left|\frac{\Delta u}{\mu}\right|^{p(x)} dx\leq 1\right\}$$
 represent a norm which is equivalent to $\|.\|_{X}$ on $X$  (see  Remark 2.1 in
 Amrouss and Ourraoui~\cite{elo}). 
 Therefore in what follows, we shall consider $\left(X, \|.\|\right).$
 The modular on the space $X$ is the map $\rho_{p(x)}:X\rightarrow \mathbb{R}$ defined by
 $$\rho_{p(x)}(u):=\int_{\Omega}|\Delta u|^{p(x)}dx.$$
 This mapping satisfies the following properties.
 \begin{lemma} (El Amrous, Moradi, Moussaoui~\cite{ela})\label{2.2} 
 For every $u, u_n \in W^{2, p(.)}(\Omega),$
 \begin{enumerate}
 \item[(1)] $\|u\|<1 \;(\mbox{resp.} =1, >1) \Longleftrightarrow \rho_{p(x)}(u)<1 \;(\mbox{resp.} =1, >1);$
 \item[(2)] $\displaystyle\min\{\|u\|^{p^{-}}, \|u\|^{p^{+}}\}\leq \rho_{p(x)}\leq\displaystyle\max\{\|u\|^{p^{-}}, \|u\|^{p^{+}}\};$  \hbox{and}
  \item [(3)] $\|u_n\|\rightarrow0\;(\mbox{resp.} \rightarrow\infty)\Leftrightarrow \rho_{p(x)}(u_n)\rightarrow0\;(\mbox{resp.} \rightarrow\infty).$
 \end{enumerate}
 \end{lemma}
 We recall that the critical Sobolev exponent is defined as follows:
 $$p^{\ast}(x):=\left\{
 \begin{array}{cc}
 \displaystyle \frac{Np(x)}{N-2p(x)}, \,\,\mbox{ if } \;p(x)<\frac{N}{2},\\
 \displaystyle \infty, \,\,\mbox{ if } \;p(x)\geq\frac{N}{2}.
 \end{array}
 \right.$$
 
 \begin{remark}\label{r1} (R\u{a}dulescu and Repov\v{s}~\cite{RR})
 \label{inj}
 Assume that $p\in C^{+}(\overline{\Omega})$  satisfies $p^-> \frac{N}{2}$.
 Then there exist a continuous embedding  $X\hookrightarrow W^{2,p^{-}}(\Omega)\cap W_{0}^{1,p^{-}}(\Omega) $ and a  compact embedding  $W^{2,p^{-}}(\Omega)\cap W_{0}^{1,p^{-}}(\Omega)\hookrightarrow C^{0}(\overline{\Omega})$, such that $X$ is compactly embedded in $C^{0}(\overline{\Omega})$ and
  $\|u\|_{\infty}\leq c_0 \|u\|,$ where $c_0$ is a positive constant and $\|u\|_{\infty}:=\sup_{x\in\Omega}|u(x)|.$
 \end{remark}
 \begin{proposition}(Gasi\'nski and Papageorgiou~\cite{GP})\label{compact}
 If $X$ is a reflexive Banach space, $Y$ is a Banach space, $Z\subset X$ is nonempty, closed and convex, and $J:Z\rightarrow Y$ is completely continuous, then $J$ is compact.
 \end{proposition}
 Our key tool will be  Theorem 3.6 in   
 Bonanno and Marano~\cite{bonano}, 
 which  we restate here in a more convenient form.
 \begin{theorem}(Bonanno and Marano~\cite{bonano})\label{bonano} 
 Let X be a reflexive real Banach space and let
 $\Phi:X\rightarrow\mathbb{R}$
 be  a coercive, continuously G\^{a}teaux differentiable and sequentially weakly lower semicontinuous functional, whose G\^{a}teaux derivative admits
 a continuous inverse on $X$.
  Let $\Psi:X\rightarrow \mathbb{R}$ be a continuously G\^{a}teaux differentiable
 functional whose G\^{a}teaux derivative is compact and such that
 $\inf_{x\in X}\Phi(X)=\Phi(0)=\Psi(0)=0.$
 Assume that there exist
  $r > 0, \overline{x}\in X$,
   such that
    $r < \Phi(\overline{x}),$
  $r^{-1}\sup_{\Phi(x)\leq r}\Psi(x)<\frac{\Psi(\overline{x})}{\Phi(\overline{x})},$
 and for each 
 $\lambda\in\Lambda_r:=\left(\frac{\Phi(\overline{x})}{\Psi(\overline{x})}, r(\sup_{\Phi(x)\leq r}\Psi(x))^{-1}\right),$
 functional
 $\Phi-\lambda \Psi$
 is coercive.
 Then for each $\lambda\in\Lambda_r$, functional $\Phi-\lambda \Psi$ has at least three distinct critical
 points in $X$.
 \end{theorem}
 
   \subsection{Conditions for parameters and nonlinearities}\label{Assump}
   In   problem $(\ref{P}),$  parameters and  nonlinearities are assumed to satisfy
 the following conditions:
 
 \begin{enumerate}
 \item[${(H_1)}$]
 $a:\overline{\Omega}\times\R\rightarrow\R$ is a Carath\'{e}odory function such that $a(x, 0)=0,$ for a.e. $x\in\Omega$.
 \item[${(H_2)}$]
 There exist  $c_1>0$ and a nonnegative function $d\in L^{\frac{p(x)}{p(x)-1}}(\Omega)$, such that
 $$|a(x, t)|\leq c_1(d(x)+|t|^{p(x)-1}), \  \mbox{for a.e.}\; x\in\Omega \ \hbox{and all} \  t\in \R.$$
 \item[${(H_3)}$]
 For all  $ s, t\in\R$, the inequality
  $|a(x, t)-a(x, s)|(t-s)\geq 0\;\;\mbox{ holds, for a.e.}\; x\in\Omega,$ with equality if and only if $s=t$.
 \item[${(H_4)}$]
 There exists $1 \leq c_2$ such that
 $$c_2|t|^{p(x)}\leq \min\{ a(x, t)t,p(x)A(x, t)\},\;\mbox{for a.e.} x\in\Omega \; \mbox{and all}\; s, t\in\R,$$
 where $c_1$ is the constant from condition  $(H_2)$ and $A :\overline{\Omega}\times\R\rightarrow\R$ represents the antiderivative of $a,$
 $$A(x, t)=\int_{0}^{t}a(x, s)ds.$$
 \item[${ {(H_5)}}$]$
 |f(x,t)|\leq \xi(x)+\zeta |t|^{q(x)-1},\ \mbox{for all} \ (x,t) \in \Omega \times \mathbb{R},$
 where $ \xi\in L^1(\Omega)$,  $\zeta$ is a positive constant, and $1< q^-\leq q^+< p^-$.
 \end{enumerate}
 
 \noindent
 \begin{remark} (Boureanu~\cite{bour})\label{c3} Concerning conditions $(H_{1})-(H_{5})$, the following can be observed:
 \begin{enumerate}
  \item[(i)] $A(x, t)$ is a $C^{1}$-Carath\'{e}odory function, i.e., for every $t\in\R,$ $A(., t):\Omega\rightarrow\R$ is measurable
 and for a.e. $x\in\Omega,$ $A(x, .)$ is $C^{1}(\R).$
 \item[(ii)]
 There exists a constant $c_3$ such that
$|A(x, t)|\leq c_3(d(x)|t|+|t|^{p(x)}), \ \mbox{for a.e.} \  x\in\Omega \;\;\mbox{and  all}\;\; t\in \R.$
 \end{enumerate}
 \end{remark}
 In order to formulate the variational approach to problem  $(\ref{P}),$
  recall that a weak solution for our problem satisfies the following definition.
 \begin{definition} We say that $u\in X\backslash\{0\}$ is a weak solution of problem $(\ref{P})$ if $\Delta u=0$ on $\partial{\Omega}$ and
 $$ \int_{\Omega}a(x, \Delta u)\Delta vdx-\lambda\int_{\Omega}f(x,u) vdx=0,\;\mbox{ for all} \  \;v\in X.$$
 \end{definition}
 Denote
 $$\Phi(u):=\int_{\Omega}F(x,u)dx.$$
 The Euler-Lagrange functional corresponding to problem $(\ref{P}),$ is defined by $\Psi_{\lambda}: X\to\mathbb{R},$ where
  $\Psi_{\lambda}(u)=J(u)- \lambda\Phi(u),\,\,\mbox{for all}\;u\in X,$
  such that $$J(u)=\int_{\Omega}A(x, \Delta u)dx.$$
  By condition $(H_5)$, we have $$ |F(x,u)|\leq \xi(x)|u|+\frac{\zeta}{q(x)}|u|^{q(x)},$$
  hence
   $$\Phi(u)\leq |\xi(x)|_{L^1(\Omega)}\|u\|_{\infty}+\frac{\zeta}{q^-}\int_{\Omega}(|u|^{q^+}+|u|^{q^-})dx.$$
  so we get
  $$\Phi(u)\leq  |\xi(x)|_{L^1(\Omega)}\|u\|_{\infty}+\frac{\zeta}{q^-}(c_{0}^{q^+}\|u\|^{q^+}+c_{0}^{q^-}\|u\|^{q^-})|\Omega|.$$
  This shows that $\Phi$ is well-defined.
 In the sequel, we shall need the following lemma.
 \begin{lemma}\label{lm}$(Boureanu~\cite{bour})$
 Functional $J$ is coercive on $X$ and $J':X\rightarrow X^*$ is a strictly monotone homeomorphism.
 \end{lemma}
 \begin{proof}
  It is clear from Lemma \ref{2.2} and hypothesis $(H_4)$ that for $u\in X$ with $\|u\| >1,$
 \begin{eqnarray}\label{coe}
 J(u)&\geq&\int_{\Omega}\frac{c_2}{p(x)}|\Delta u|^{p(x)}dx
 \geq \frac{1}{p^+}\rho_{p(x)}(u)
 \geq  \frac{1}{p^+} \|u\|^{p^-},
 \end{eqnarray}
 and thus $J$ is coercive.
 For the rest of the proof of Lemma \ref{lm} see
 Boureanu~\cite{bour}.
 \end{proof}
  Next, we shall show
  that
   $\Phi'(u)$ is compact. Let $v_n\rightharpoonup v$ in $X.$ Remark \ref{inj} asserts that $v_n\rightarrow v$ in $C_{0}(\overline\Omega)$, so for any $u\in X$,
 \begin{eqnarray*}
 |<\Phi'(u),v_n>|-|<\Phi'(u),v>|&\leq& \|v_n-v\|_{\infty}\int_{\Omega} f(x,u)dx.
 \end{eqnarray*}
 As a consequence of condition $(H_5)$, we have
 $$|<\Phi'(u),v_n>|\rightarrow|<\Phi'(u),v>|, \ \hbox{ as} \ n\rightarrow +\infty.$$
  This means that $\Phi'(u)$ is completely continuous, hence compact by  Proposition \ref{compact}.
 
 \section{The Main Result}\label{s3}
 
 Setting $\delta(x):=\sup\{\delta>0\mid B(x,\delta)\subseteq\Omega\},$
 for all
 $x\in\Omega,$
  one can prove that there exists $x_0\in\Omega$ such that $B(x_0,D)\subseteq\Omega$,
 where $D:=\sup_{x\in\Omega}\delta(x).$
 For each $r > 0$, let 
  $\gamma_r:=\max\{(p^{+} r)^{\frac{1}{p^+}},(p^+ r)^{\frac{1}{p^-}}\}.$
 Let
 $$L:=w(D^N-(\frac{D}{2})^N),
 \
 \hbox{where}
 \
  w:=\frac{\pi^{\frac{N}{2}}}{\frac{N}{2}\Gamma(\frac{N}{2})},$$
 and $\Gamma$ denotes the Euler function. We can now state the main result of this paper.
 \begin{theorem}\label{principal}\textbf{}
 Let $a:\Omega\times\mathbb{R}^N\rightarrow R^N$ be  potential satisfying conditions $(H_2)-(H_4)$
 and
   $f:\Omega\times \mathbb{R}\rightarrow \mathbb{R}$  a Carath\'eodory  function satisfing condition $(H_5)$ and the following requirement
  $$\essinf_{x\in\Omega}F(x,t):=\essinf_{x\in\Omega}\displaystyle\int_{0}^{t}f(x,s)ds\geq 0,
  \
  \hbox{for all}
  \  \
  t\in[0,h],$$ where h is a nonnegative constant.
 Suppose that there exist $r>0,h>0$ such that
 $$
 r<\frac{L}{p^+}\min\Big\{\Big(\frac{8hN}{3D^2}\Big)^{p^-},\Big(\frac{8hN}{3D^2}\Big)^{p^+}\Big\},
 $$
 \begin{equation}\label{r}
 \beta_h := \frac{w(\frac{D}{2})^N \essinf_{x\in\Omega}F(x,h)}{c_3 
 L^{\frac{1}{p^+}}\Big[N^{\frac{1}{p^+}}\frac{8h}{3D^2} |d(x)|_{\frac{p(x)}{p(x)-1}}
 + L^{\frac{p^{+}-1}{p^+}}\max\big\{(\frac{8hN}{3D^2})^{p^-},
 (\frac{8hN}{3D^2})^{p^+}\big\}\Big]}
 \end{equation} 
  $$
 >\alpha_r:=\frac{1}{r}\int_{\Omega}\sup_{|t|\leq c_0\gamma_r}F(x,t)dx.
 $$
  Then for every $\lambda\in\overline{\Lambda}:=(\frac{1}{\beta_h}, \frac{1}{\alpha_r}),$ problem \eqref{P} admits at least three weak solutions.
 \end{theorem}
 \begin{proof}
  We shall apply Theorem \ref{bonano}. Let $X=W^{2,p(x)}(\Omega)\cap W_{0}^{1,p(x)}(\Omega)$
  and
 $\Psi_{\lambda}(u):= J(u)-\lambda\Phi(u)$, for all $u \in X$, where
 $$J(u)=\int_{\Omega}A(x, \Delta u)dx
 \ \
 \hbox{and}
 \ \
 \phi(u)=\int_{\Omega}F(x,u)dx.$$
 As we have seen before, functionals $J$ and $\Phi$ satisfy the regularity assumptions
 of Theorem \ref{bonano}.
 Now let $\overline{v}\in X$ be defined by
 \begin{equation*}
 \overline{v}=\left\{
 \begin{array}{ccc}
 &0,& \mbox{if}\; x\in\Omega\setminus B(x_0,D)\\
 &h,&   \mbox{if}\; x\in B(x_0,\frac{D}{2})\\
 &\frac{4h}{3D^2}(D^2-|x-x_0|^2),& \quad \quad \mbox{if}\; x\in B(x_0,D)\setminus B(x_0,\frac{D}{2}),
 \end{array}
 \right.
 \end{equation*}
 where $|.|$ denotes the Euclidean norm in $\mathbb{R}^N$, and
 \begin{equation*}
 \bigtriangleup\overline{v}=\left\{
 \begin{array}{cc}
 0,\quad  \quad  \mbox{if} \; x\in\Omega\setminus B(x_0,D)\cup B(x_0,\frac{D}{2})\\
 [2mm]
 \frac{-8h}{3D^2}N, \quad \quad \mbox{if}\;   x\in B(x_0,D)\setminus B(x_0,\frac{D}{2}).
 \end{array}
 \right.
 \end{equation*}
 Using the above information, Remark \ref{r1}, Lemma \ref{2.2}, and the continuity of
 embedding $L^{p^+}(\Omega)\hookrightarrow L^{p(x)}(\Omega)$, we obtain
 \begin{eqnarray*}
 \hspace{-1cm}&&\frac{L}{p^+}\min\big\{(\frac{8hN}{3D^2})^{p^-},(\frac{8hN}{3D^2})^{p^+}\big\}\leq J(\overline{v})\\
 &\leq& c_3 L^{\frac{1}{p^+}}\Big[N^{\frac{1}{p^+}}\frac{8h}{3D^2} |d(x)|_{\frac{p(x)}{p(x)-1}}
 +L^{\frac{p^{+}-1}{p^+}} \max\big\{(\frac{8hN}{3D^2})^{p^-},(\frac{8hN}{3D^2})^{p^+}\big\}\Big],
 \end{eqnarray*}
 and $$\Phi(\overline{v})\geq\displaystyle\int_{B(x_0,\frac{D}{2})}F(x,\overline{v}(x)dx\geq w(\frac{D}{2})^N\essinf_{x\in\Omega}F(x,h).$$
 It follows from \eqref{r}  that $r<J(\overline{v})$.
 Moreover, since the embedding $X\hookrightarrow C_{0}(\overline{\Omega})$ is continuous,  we have
 $$\max_{x\in\Omega}|u(x)|\leq c_0\max\big\{(p^{+} r)^{\frac{1}{p^+}},(p^+ r)^{\frac{1}{p^-}}\big\}=c_0\gamma_r,
 \
 \hbox{for all}
 \
 u\in X,
 \
 \hbox{such that}
 \
 J(u)\leq r.
 $$
 Therefore $$\sup_{J(u)\leq r}\Phi(u)\leq\displaystyle\int_{\Omega}\sup_{|t|\leq c_0\gamma_r}F(x,t)dx.$$
 By the definitions of $\alpha_r$ and $\beta_h$ in Theorem \ref{principal}, we obtain $$\frac{\sup_{J(u)\leq r}\Phi(u)}{r}\leq \alpha_r<\beta_h\leq \frac{\Phi(\overline{v})}{J(\overline{v})}.$$
 Hence the first condition of Theorem \ref{bonano} has been verified.
 
 Next, we shall prove that for each $\lambda>0$, the energy functional $J-\lambda\Phi$ is coercive. For any $q\in C_+(\overline{\Omega})$, we have $|u(x)|^{q(x)}\leq |u(x)|^{q^+}+|u(x)|^{q^-}.$
 Condition $(H_5),$ Remark~\ref{inj}, and the above inequality imply
 \begin{eqnarray*}
 \Phi(u)&\leq&|\xi(x)|_{L^1(\Omega)}\|u\|_{\infty}+\frac{\alpha}{q^-}\int_{\Omega}(|u|^{q^+}+|u|^{q^-})dx\\
 &\leq& |\xi(x)|_{L^1(\Omega)}\|u\|_{\infty}+\frac{\alpha}{q^-}\big(\|u\|_{+\infty}^{q^-}+(\|u\|_{+\infty}^{q^+}\big)|\Omega|\\
 &\leq&  |\xi(x)|_{L^1(\Omega)}\|u\|_{\infty}(c_{0}^{q^+}\|u\|^{q^+}+c_{0}^{q^-}\|u\|^{q^-})|\Omega|.
 \end{eqnarray*}
 The above inequality and  \eqref{coe} give
 $$J(u)-\lambda\Phi(u)\geq \frac{1}{p^+} \|u\|^{p^-}-c_{0} |\xi(x)|_{L^1(\Omega)}\|u\|+\frac{\alpha}{q^-}(c_{0}^{q^+}\|u\|^{q^+}+c_{0}^{q^-}\|u\|^{q^-})|\Omega|.$$
 Since $1\leq q^-\leq q^+<p^-$, it follows that $J(u)-\lambda\Phi(u)$ is coercive.
 
 Finally, we use the fact that 
 $$\overline{\Lambda}:=
 \left(\frac{1}{\beta_h}, \frac{1}{\alpha_r}\right)\subseteq\left(\frac{J(\overline{v})}{\Phi(\overline{v})},\frac{r}{\sup_{J(u)\leq r}\Phi(u)}\right).$$
 Theorem~\ref{bonano} now ensures that for each $\lambda\in \overline{\Lambda}$,  functional  $J(u)-\lambda\Phi(u)$ admits at least three critical points in X which are weak solutions for problem \eqref{P}.
 This completes the proof of Theorem \ref{principal}.
 \end{proof}
 \begin{remark}
 Setting $r=1$ in Theorem \ref{principal}, we get
  $\gamma_r=(p^{+})^{\frac{1}{p^-}}$ and inequalities \eqref{r} become
 $$p^+<L\min\big\{(\frac{8hN}{3D^2})^{p^-},(\frac{8hN}{3D^2})^{p^+}\big\},$$  and
 $$\beta_h:=
 \frac{w(\frac{D}{2})^N\essinf_{x\in\Omega}F(x,h)}{c_3 L^{\frac{1}{p^+}}\Big[N^{\frac{1}{p^+}}\frac{8h}{3D^2} |d(x)|_{\frac{p(x)}{p(x)-1}}
 +L^{\frac{p^{+}-1}{p^+}} \max\big\{(\frac{8hN}{3D^2})^{p^-},(\frac{8hN}{3D^2})^{p^+}\big\}\Big]}
 $$
 $$> \alpha:=\int_{\Omega}\sup_{|t|\leq c_0(p^{+})^{\frac{1}{p^-}}}F(x,t)dx.
 $$
 \end{remark}
 \begin{remark}
 We note that, if $f(x,0)\neq 0$, then by Theorem \ref{principal}, there exist at least three
 nonzero solutions.
 \end{remark}
 \begin{remark}\label{ope}
 We are interested in the Leray-Lions type operators because they are quite general. Indeed, consider
 \begin{equation}\label{a}
 a(x,t)=\theta(x)|t|^{p(x)-2}t.
 \end{equation}
 where $p(x)$ satisfies condition
 \eqref{e2.2} and for $\theta\in L^{\infty}(\Omega)$  there exists $\theta_0>0$ with $\theta(x)\geq\theta_0>0,$
 for a.e. $x\in\Omega$. One can see that equation \eqref{a} satisfies conditions $(H_1)-(H_4)$ and we arrive at  operator
 $\theta(.)\Delta^2\big(|\Delta|^{p(.)-2} u\big).$
 
 Note that when $\theta\equiv 1$, we  get the well-known  p(x)-biharmonic operator $\Delta_{p(.)}^2(u)$.
 Moreover, we can make the following choice $a(x,t)=\theta(x)(1+|t|^2)^{\frac{p(x)}{p(x)-2}}t,$
 and obtain  operator $\theta(.)\Delta\Big((1+|\Delta u|^2)^{\frac{p(.)}{p(.)-2}}\Delta u\Big),$
 where $p$ and $\theta$ are as above.
 \end{remark}
  
  \section{An application}\label{s4}
   We present
   an 
   interesting application of Theorem \ref{principal}. Let $\alpha:[0,1]\rightarrow\mathbb{R}$, be a positive, bounded and measurable function. Put $\alpha_0=\essinf_{x\in  [0,1]}\alpha(x)$ and $\|\alpha\|_1=\|\alpha\|_{L^{1}([0,1])}$. Moreover, set $$k=\frac{1}{p^{+}c_3}(\frac{3}{8})^{p^{+}}\frac{\alpha_0}{\|\alpha\|_1},$$
   where $c_3$ is the constant  from Remark \ref{c3}.
   \begin{theorem}\label{dim1}
   Let $g:\mathbb{R}\rightarrow\mathbb{R}$ be a nonnegative continuous function such that
   $\lim_{|t|\rightarrow +\infty}{g(t)}{|t|^{-\nu}}=0,$
   for some $0\leq\nu<p^--1$ and $g(0)\neq 0$. Let $G(\xi)=\int_{0}^{\xi}g(t)dt,$  for all $\xi\in\mathbb{R},$ and assume that there exist
   positive 
   constants $h$ and $l$, with $l\leq 1\leq (\frac{8}{3}h)^{\frac{p^-}{p^+}}(\frac{1}{4})^{\frac{1}{p^+}}$,  such that
   $\frac{G(l)}{l^{p^+}}<k\frac{G(h)}{h^{p^+}}.$
   Then for every
    $\lambda\in\Big(\frac{(\frac{8}{3})^{p^+}h^{p^+}c_3}{\alpha_0G(h)},\frac{l^{p^+}}{p^+\|\alpha\|_1 G(l)}\Big),$  the following problem
   \begin{equation}\label{res}
   \left\{
 \begin{array}{ccc}
 (|u^{''}|^{p(x)-2}u^{''})^{''} = \lambda \alpha(x) g(u)\, \ \hbox{in}\, \  (0,1),\\
  u(0)=u(1)=0,\\
  u^{''}(1)=u^{''}(0)=0
 \end{array}
 \right.
 \end{equation}
 has at least three nontrivial solutions.
   \end{theorem}
   \begin{proof}
   For each $u\in X$ and $x\in[0,1]$, we have $$u(x)=\int_{0}^{x}u'(t)dt=\int_{1}^{x}u'(t)dt$$
   and
   \begin{eqnarray}
   |u(x)|&=& \frac{1}{2}\Big(\left|\int_{0}^{x}u'(t)dt\right|+\left|\int_{1}^{x}u'(t)dt\right|\Big)
   =\frac{1}{2}\Big[\int_{0}^{x}|u'(t)|dt+\int_{1}^{x}|u'(t)dt|\Big]
   =\frac{1}{2}\int_{0}^{1}u'(t)dt.
   \end{eqnarray}
   For each $u\in C^2([0,1])$ there exists $\eta\in(0,1)$ such that $u'(\eta)=0$. Hence, one has
   $$u'(t)=\int_{\eta}^{t}u''(s)ds, \ \mbox{for all} \ t\in (0,1).$$
 By the H\"older inequality \eqref{ing}, we can conclude that 
 $$|u'(t)|\leq\int_{0}^{1}|u''(s)|ds\leq\frac{1}{2} |u^{''}|_{L^{p(x)}([0,1])}|1|_{L^{q(x)}([0,1])},\quad 0<\eta\leq t\leq1,$$
 where $\frac{1}{p(x)}+\frac{1}{q(x)}=1$. Since $|1|_{L^{q(x)}([0,1])}\leq 1$, we see that
 $$|u(x)|\leq \frac{1}{4}\|u\|,
 \
 \hbox{ for all}
 \
 t\in[0,1],
 u\in W^{2,p(x)}([0,1])\cap W_{0}^{1,p(x)}([0,1]),
 $$ so $c_0\leq1$, where $c_0$ is the constant from Remark \ref{inj}. If we take
  $r=\frac{l^{p^+}}{p^+}$, we get $\gamma_r=l$, and a simple computation shows that all hypotheses of Theorem~\ref{principal} hold and thus it can be applied. This completes the proof of Theorem~\ref{dim1}.
 \end{proof}
 \begin{remark}
 Our Theorem \ref{dim1} generalizes Theorem 3.3 in Bonanno and Chinn\`{i} \cite{BoCh1}.
  \end{remark}  
  
\section*{Acknowledgements}
The second author was supported by the Slovenian Research Agency program P1-0292 and grants N1-0114 and N1-0083.

\hfill
 
Faculty of Computer Science and Information Technology,  Northern Border University, Rafha, Kingdom of Saudi Arabia,
 {\it khaled\_kefi@yahoo.fr} \\
 https://orcid.org/0000-0001-9277-5820   \\

Faculty of Education and Faculty of Mathematics and Physics, University of Ljubljana \&
Institute of Mathematics, Physics and Mechanics,  1000 Ljubljana, Slovenia,
{\it dusan.repovs@guest.arnes.si} \\
https://orcid.org/0000-0002-6643-1271   \\

College of Sciences,  Imam Abdulrahman Bin Faisal University,
 31441  Dammam, Kingdom of Saudi Arabia,
{\it kmsaoudi@iau.edu.sa}\\
https://orcid.org/0000-0002-7647-4814
\end{document}